\documentclass[oneside,english]{amsart}
\usepackage[T1]{fontenc}
\usepackage[latin9]{inputenc}
\usepackage{mathrsfs}
\usepackage{amstext}
\usepackage{amsthm}
\usepackage{amssymb}
\usepackage[all]{xy}

\makeatletter
\numberwithin{equation}{section}
\numberwithin{figure}{section}
\theoremstyle{plain}
\newtheorem{thm}{\protect\theoremname}
\theoremstyle{remark}
\newtheorem{rem}[thm]{\protect\remarkname}
\theoremstyle{plain}
\newtheorem{prop}[thm]{\protect\propositionname}
\theoremstyle{plain}
\newtheorem{cor}[thm]{\protect\corollaryname}
\theoremstyle{plain}
\newtheorem{lem}[thm]{\protect\lemmaname}
\theoremstyle{definition}
\newtheorem{example}[thm]{\protect\examplename}

\makeatother

\usepackage{babel}
\providecommand{\corollaryname}{Corollary}
\providecommand{\examplename}{Example}
\providecommand{\lemmaname}{Lemma}
\providecommand{\propositionname}{Proposition}
\providecommand{\remarkname}{Remark}
\providecommand{\theoremname}{Theorem}

\begin{document}
\global\long\def\G{\mathcal{G}}%
 
\global\long\def\F{\mathcal{F}}%
 
\global\long\def\H{\mathcal{H}}%
\global\long\def\Z{\mathcal{Z}}%
 
\global\long\def\L{\mathcal{L}}%
\global\long\def\U{\mathcal{U}}%
\global\long\def\W{\mathcal{W}}%
 
\global\long\def\E{\mathcal{E}}%
\global\long\def\B{\mathcal{B}}%
 
\global\long\def\A{\mathcal{A}}%
\global\long\def\D{\mathcal{D}}%
\global\long\def\O{\mathcal{O}}%
 
\global\long\def\N{\mathcal{N}}%
 
\global\long\def\X{\mathcal{X}}%
 
\global\long\def\lm{\lim\nolimits}%
 
\global\long\def\then{\Longrightarrow}%
\global\long\def\BaseD{\operatorname{B}_{\mathbb{D}}}%

\global\long\def\V{\mathcal{V}}%
\global\long\def\C{\mathcal{C}}%
\global\long\def\R{\operatorname{R}}%
\global\long\def\adh{\operatorname{adh}}%
\global\long\def\Seq{\operatorname{Seq}}%
\global\long\def\intr{\operatorname{int}}%
\global\long\def\cl{\operatorname{cl}}%
\global\long\def\inh{\operatorname{inh}}%
\global\long\def\id{\operatorname{id}}%
\global\long\def\diam{\operatorname{diam}\ }%
\global\long\def\card{\operatorname{card}}%
\global\long\def\T{\operatorname{T}}%
\global\long\def\S{\operatorname{S}}%
\global\long\def\I{\operatorname{I}}%
\global\long\def\AdhD{\operatorname{A}_{\mathbb{D}}}%
\global\long\def\UD{\operatorname{U}_{\mathbb{D}}}%
\global\long\def\K{\operatorname{K}}%
\global\long\def\LD{\operatorname{L}_{\mathbb{D}}}%
\global\long\def\BRD{\operatorname{B}_{\mathcal{R}(\mathbb{D})}}%
\global\long\def\Dis{\operatorname{Dis}}%
\global\long\def\Cha{\operatorname{Cha}}%
\global\long\def\Reg{\operatorname{Reg}}%

\global\long\def\fix{\operatorname{fix}}%
\global\long\def\Epi{\operatorname{Epi}}%

\global\long\def\cont{\mathscr{C}}%
\global\long\def\conv{\mathsf{Conv}}%
\global\long\def\prtop{\mathsf{PrTop}}%
\global\long\def\Top{\mathsf{Top}}%
\global\long\def\pstop{\mathsf{PsTop}}%

\title{On connected subsets of a convergence space}
\author{Bryan Castro Herrejón and Frédéric Mynard}
\email{bryan\_ch@ciencias.unam.mx}
\email{fmynard@njcu.edu}
\begin{abstract}
Though a convergence space is connected if and only if its topological
modification is connected, connected subsets differ for the convergence
and for its topological modification. We explore for what subsets
connectedness for the convergence or for the topological modification
turn out to be equivalent. In particular, we show that the largest
set containing a given connected set for which all subsets in between
are connected is the adherence of the connected set for the reciprocal
modification of the convergence.
\end{abstract}

\maketitle

\section{preliminaries on convergence spaces}

If $X$ is a set, let $\mathbb{P}X$ denote its powerset, $\mathbb{F}X$
denote the set of (set-theoretic) filters on $X$, $\mathbb{U}X$
denote the set of ultrafilters on $X$ and $\mathbb{F}_{0}X$ denote
the set of principal filters on $X$. Given $\A\subset\mathbb{P}X$,
let $\A^{\uparrow}=\{B\subset X:\exists A\in\A,A\subset B\}$ and
$\A^{\#}=\{H\subset X:\forall A\in\A,A\cap H\neq\emptyset\}$. If
$\F\in\mathbb{F}X$, let $\beta\F=\{\U\in\mathbb{U}X:\F\subset\U\}$;
we also write $\beta A$ for $\beta\{A\}^{\uparrow}$.

Convergence spaces form a useful generalization of topological spaces.
Namely, a \emph{convergence }$\xi$ on a set $X$ is a relation between
points of $X$ and (set-theoretic) filters on $X$, denoted $x\in\lm_{\xi}\F$
if $(x,\F)\in\xi$ and interpreted as $x$ is a limit point for $\F$
in $\xi$, satisfying
\begin{equation}
\tag{centered}x\in\lm_{\xi}\{x\}^{\uparrow}\label{eq:centeredconv}
\end{equation}
for every $x\in X$ and 
\begin{equation}
\tag{monotone}\F\leq\G\then\lm_{\xi}\F\subset\lm_{\xi}\G,\label{eq:convmontone}
\end{equation}
for all pair $\F,\G\in\mathbb{F}X$. The pair $(X,\xi)$ is then called
a \emph{convergence space.}

A map $f:X\to Y$ between two convergence spaces $(X,\xi)$ and $(Y,\tau)$
is \emph{continuous}, in symbols $f\in\cont(\xi,\tau)$, if $f(x)\in\lm_{\tau}f[\F]$
whenever $x\in\lm_{\xi}\F$, where $f[\F]=\{f(F):F\in\F\}^{\uparrow}$
is the image filter. Let $\conv$ denote the category of convergence
spaces and continuous maps. If $\xi,\sigma$ are two convergences
on $X$, we say that $\xi$ is \emph{finer than $\sigma$ }or that
$\sigma$ is \emph{coarser than $\xi$}, in symbols $\xi\geq\sigma$,\emph{
}if the identity map $\id_{X}\in\cont(\xi,\sigma)$, that is, if $\lim_{\xi}\F\subset\lim_{\sigma}\F$
for every filter $\F$ on $X$. The set of convergences on a given
set is a complete lattice for this order. Moreover, $\conv$ is a
topological category (see \cite{adamek1990abstract} for details),
that is, it has all initial (and final) structures. Indeed, if $f:X\to(Y,\tau)$
there is the coarsest convergence $f^{-}\tau$ on $X$ making $f$
continuous (to $(Y,\tau)$), namely, $x\in\lm_{f^{-}\tau}\F$ if $f(x)\in\lim_{\tau}f[\F]$,
and it is initial. If $f_{i}:X\to(Y_{i},\tau_{i})$ for $i\in I$,
the supremum $\bigvee_{i\in I}f_{i}^{-}\tau_{i}$ provides an initial
lift for the source. As a result, $\conv$ has products and subspaces.
In particular if $(X,\xi)$ is a convergence space and $A\subset X$,
the \emph{induced convergence }or \emph{subspace convergence }$\xi_{|A}$
on $A$ is the initial convergence $i^{-}\xi$ for the inclusion map
$i:A\to X$, that is, $a\in\lm_{\xi_{|A}}\F$ if $a\in\lm_{\xi}\F^{\uparrow_{X}}$.

If a convergence additionally satisfies
\begin{equation}
\tag{finite depth}\lm_{\xi}(\F\cap\G)=\lm_{\xi}\F\cap\lm_{\xi}\G\label{eq:fintedepth}
\end{equation}
 for every $\F,\G\in\mathbb{F}X$, we say that $\xi$ \emph{has finite
depth} (\footnote{Many authors include (\ref{eq:fintedepth}) in the axioms of a convergence
space. Here we follow \cite{DM.book}.}). 

A convergence satisfying the stronger condition that 
\begin{equation}
\tag{pretopology}\lm_{\xi}(\bigcap_{\D\in\mathbb{D}}\D)=\bigcap_{\D\in\mathbb{D}}\lm_{\xi}\D\label{eq:pretopdeep}
\end{equation}
 for every $\mathbb{D}\subset\mathbb{F}X$ is called a \emph{pretopology}
(and $(X,\xi)$ is a \emph{pretopological space}). 

A subset $A$ of a convergence space $(X,\xi)$ is \emph{$\xi$-open
}if 
\[
\lm_{\xi}\F\cap A\neq\emptyset\then A\in\F,
\]
and $\xi$-\emph{closed} if it is closed for limits, that is,
\[
A\in\F\then\lm_{\xi}\F\subset A.
\]

Let $\O_{\xi}$ denote the set of open subsets of $(X,\xi)$ and let
$\O_{\xi}(x)=\{U\in\O_{\xi}:x\in U\}$. Similarly, let $\mathcal{C}_{\xi}$
denote the set of closed subsets of $(X,\xi)$. It turns out that
$\O_{\xi}$ is a topology on $X$\emph{. }Moreover,\emph{ }a topology
$\tau$ on a set $X$ determines a convergence $\xi_{\tau}$ on $X$
by
\[
x\in\lm_{\xi_{\tau}}\F\iff\F\geq\N_{\tau}(x),
\]
where $\N_{\tau}(x)$ denotes the neighborhood filter of $x$ for
$\tau$. In turn, $\xi_{\tau}$ completely determines $\tau$ because
$\tau=\O_{\xi_{\tau}}$, so that we do not distinguish between $\tau$
and $\xi_{\tau}$ and identify topologies with special convergences.
Moreover, a convergence $\xi$ on $X$ determines the topology $\O_{\xi}$
on $X$ which turns out to be the finest among the topologies on $X$
that are coarser than $\xi$. We call it the \emph{topological modification
of $\xi$ }and denote it $\T\xi$. The map $\T$ turns out to be a
concrete reflector. Hence the category $\Top$ of topological spaces
and continuous maps is a concretely reflective subcategory of $\conv$.
Let $\cl_{\xi}$ denote the closure operator in $\T\xi$.

Just like in the topological case, we will say that a convergence
space is $T_{1}$ if every singleton is closed, and \emph{Hausdorff
}if limit sets have at most one element.

Consider the \emph{principal adherence operator }$\adh_{\xi}:\mathbb{P}X\to\mathbb{P}X$
given by
\[
\adh_{\xi}A=\bigcup_{A\in\G^{\#}}\lm_{\xi}\G=\bigcup_{A\in\G}\lm_{\xi}\G=\bigcup_{\U\in\beta A}\lm_{\xi}\U.
\]

In general,
\[
\adh_{\xi}A\subset\cl_{\xi}A
\]
but not conversely. In contrast to $\cl_{\xi}$, the principal adherence
is in general non-idempotent because $\adh_{\xi}A$ need not be closed.

In view of (\ref{eq:pretopdeep}), a convergence $\xi$ is a pretopology
if and only if the \emph{vicinity filter }$\V_{\xi}(x)=\bigcap_{x\in\lm_{\xi}\F}\F$
converges to $x$, for every $x\in X$, equivalently if it is determined
by its principal adherence via $\lim_{\xi}\F=\bigcap_{A\in\F^{\#}}\adh_{\xi}A$
(\footnote{To see this equivalence, note that 
\[
A\in(\V_{\xi}(x))^{\#}\iff x\in\adh_{\xi}A.
\]
}). The full subcategory $\prtop$ of $\conv$ formed by pretopological
spaces and continuous maps is concretely reflective, with reflector
$\S_{0}$ given on objects by
\[
\lm_{\S_{0}\xi}\F=\bigcap_{A\in\F^{\#}}\adh_{\xi}A.
\]
On the other hand, the reflector $\T$ is given by
\[
\lm_{\T\xi}\F=\bigcap_{A\in\F^{\#}}\cl_{\xi}A,
\]
so that $\T\leq\S_{0}$, that is, every topology is a pretopology.
In fact, a pretopology is a topology if and only if the principal
adherence operator is idempotent, in which case $\adh_{\xi}=\cl_{\xi}$.\emph{
}We refer the reader to \cite{DM.book} for a systematic study of
convergence spaces and their applications to topological problems.

\section{Connectedness}

By definition, a convergence space $(X,\xi)$ is \emph{connected }if
its only subsets that are both closed and open (clopen) are $\emptyset$
and $X$, equivalently if every continuous map $f:X\to\{0,1\}$ (where
$\{0,1\}$ carries the discrete topology) is constant.

As a result, it is plain that $\xi$ is connected if and only if its
topological modification $\T\xi$ is, so that the theory of connectedness
for convergence spaces is mostly equivalent to the classical one for
topological spaces. However, some care is needed for subspaces, as
well as for products. We say that a subset $A$ of a convergence space
$(X,\xi)$ is \emph{connected }if it is as a subspace, that is, if
$(A,\xi_{|A})$ is connected. For the reason already outlined, this
is equivalent to $(A,\T(\xi_{|A}))$ being connected, but the subtlety
lies in the fact that while
\begin{equation}
\T(\xi_{|A})\geq(\T\xi)_{|A},\label{eq:Tsubspace}
\end{equation}
the reverse inequality does not need to hold, e.g., \cite[Example V.4.35]{DM.book}.
Note also that if $\xi$ and $\tau$ are convergences on $X$ with
$\xi\geq\tau$, then every $\xi$-connected subset of $X$ is also
$\tau$-connected. In particular, every $\xi$-connected subset is
also $\T\xi$-connected, but the converse is not true, e.g., \cite[Example XII.1.30]{DM.book}. 
\begin{rem}
Similarly $\T$ does not commute with products, so some care is needed
in extending results on the connectedness of products from topological
to convergence spaces, but the fact that a product space is connected
if and only if each factor is connected remains valid for convergence
spaces. Indeed, though \cite[Theorem XII.1.31]{DM.book} is formulated
for topological spaces, its proof can easily be adapted to extend
the result to arbitrary convergence spaces. See, e.g., \cite[Teorema 6.2.3]{bryantesis}
for a complete proof.
\end{rem}

On the other hand, 
\begin{prop}
\label{prop:pretopconnect} A subset $A$ of a convergence space $(X,\xi)$
is $\xi$-connected if and only if it is $\S_{0}\xi$-connected.
\end{prop}

\begin{proof}
Since $\xi\geq\S_{0}\xi$, every $\xi$-connected subset is also $\S_{0}\xi$-connected.
Conversely, if $A$ is $\S_{0}\xi$-connected, that is, $(\S_{0}\xi)_{|A}$
is connected, then so is $\S_{0}(\xi_{|A})$ because $\S_{0}(\xi_{|A})=(\S_{0}\xi)_{|A}$
by \cite[Corollary XIV.3.9]{DM.book}. If $A$ is connected for $\S_{0}(\xi_{|A})\geq\T(\xi_{|A})$,
it is $\T(\xi_{|A})$-connected, equivalently, $\xi$-connected.
\end{proof}
In other words, in studying connected subsets of a convergence space,
we can restrict ourselves to pretopologies.

The extension to convergence spaces of the topological fact that if
$A\subset X$ is connected and $A\subset B\subset\cl A$ then $B$
is also connected is:
\begin{prop}
\cite[Prop. XII.1.25]{DM.book} \label{prop:adhsandwich}If $A$ is
a $\xi$-connected subset of $|\xi|$ and $A\subset B\subset\adh_{\xi}A$
then $B$ is also $\xi$-connected.
\end{prop}

\cite{DM.book} is however silent on whether this result would extend
if $\adh_{\xi}A$ is replaced by $\cl_{\xi}A$ even if $\xi$ is not
topological. The first observation is that:
\begin{prop}
\label{prop:closedconn}If $S$ is a $\xi$-connected subset of $|\xi|$
then $\cl_{\xi}S$ is connected.
\end{prop}

\begin{proof}
If $S$ is $\xi$-connected, it is also $\T\xi$-connected by \cite[Prop. XII.1.8]{DM.book},
and thus so is $\cl_{\xi}S$ by the usual topological result. By \cite[Proposition V.4.36]{DM.book}
by or Lemma \ref{lem:openorclosedT} below, for $A=\cl_{\xi}S$, we
have $\T(\xi_{|\cl S})=(\T\xi)_{|\cl S}$. Hence $\cl_{\xi}S$ is
$\T(\xi_{|\cl S})$-connected, equivalently $\xi$-connected. 
\end{proof}
Hence, if $A$ is connected and
\[
A\subset B\subset\cl_{\xi}A,
\]
 we can conclude that $B$ is connected when $B$ is closed (because
$B=\cl_{\xi}A$) or when $B\subset\adh_{\xi}A$, but is it true in
general? We show that the answer is no in general and explore when
it is yes.

We will use the fact that the classical result that the union of a
family of connected subspaces with non-empty intersection is again
connected extends from topological to convergence spaces:
\begin{prop}
\label{prop:unionconnected}\cite[Corollary XX.1.22]{DM.book} If
$\{A_{t}:t\in T\}$ is a collection of connected subspaces of a convergence
space $(X,\xi)$ and $\bigcap_{t\in T}A_{t}\neq\emptyset$, then $\bigcup_{t\in T}A_{t}$
is also a connected subspace.
\end{prop}

We can refine Proposition \ref{prop:closedconn}. To this end, recall
that we can define iterations of the adherence of a set by transfinite
induction by $\adh_{\xi}^{0}A=A$ and if $\adh_{\xi}^{\beta}A$ has
been defined for all $\beta<\alpha$, then 
\[
\adh_{\xi}^{\alpha}A=\adh_{\xi}\big(\bigcup_{\beta<\alpha}\adh_{\xi}^{\beta}A\big).
\]
Of course, $\adh_{\xi}^{\alpha}A\subset\cl_{\xi}A$ for every ordinal
$\alpha$ and there is equality for $\alpha$ sufficiently large.
The smallest ordinal $\alpha$ such that $\adh_{\xi}^{\alpha}A=\cl_{\xi}A$
for every subset $A$ of $(X,\xi)$ is called the \emph{topological
defect of }$\xi$. Hence, pretopologies of topological defect 1 are
topologies. 

\begin{prop}
\label{prop:iterated} If $A$ is a connected subspace of a convergence
space $(X,\xi)$ then $\adh_{\xi}^{\alpha}A$ is connected for every
ordinal $\alpha$.
\end{prop}

\begin{proof}
This is clear for $A=\emptyset$, so assume $A\neq\emptyset.$ We
proceed by transfinite induction. The case of $\alpha=0$ is clear
because $A$ is connected. Suppose $\adh_{\xi}^{\beta}A$ is connected
for all $\beta<\alpha$. In view of Proposition \ref{prop:unionconnected},
$\bigcup_{\beta<\alpha}\adh_{\xi}^{\beta}A$ is connected and thus
$\adh_{\xi}^{\alpha}A=\adh_{\xi}(\bigcup_{\beta<\alpha}\adh_{\xi}^{\beta}A)$
is connected by Proposition \ref{prop:adhsandwich}.
\end{proof}
In view of Proposition \ref{prop:pretopconnect},
\begin{prop}
\label{prop:S0xiTxi} If $(X,\xi)$ is a convergence space with topological
defect 1, equivalently, $\S_{0}\xi=\T\xi$, then $\xi$-connected
subsets and $\T\xi$-connected subsets coincide.
\end{prop}

To illustrate this proposition, recall from \cite{why} that a topological
space $(X,\tau)$ is \emph{of accessibility, }or is an \emph{accessibility
space,} if for each $x_{0}\in X$ and every $H\subset X$ with $x_{0}\in\cl_{\tau}(H\setminus\{x_{0}\}),$
there is a closed subset $F$ of $X$ with $x_{0}\in\cl(F\setminus\{x_{0}\})$
and $x_{0}\notin\cl_{\tau}(F\setminus H\setminus\{x_{0}\})$. 

Following \cite{DG}, we say that a pretopology $\tau$ is \emph{topologically
maximal within the class of pretopologies }if 
\[
\sigma=\S_{0}\sigma\geq\tau\text{ and }\T\sigma=\T\tau\then\sigma=\tau.
\]

It is proved in \cite{DG,dolecki1998topologically} that a topology
is accessibility if and only if it is topologically maximal within
the class of pretopologies. Hence, if $\xi$ is a convergence for
which $\T\xi$ is an accessibility space, then $\S_{0}\xi=\T\xi$
by maximality, so that Proposition \ref{prop:S0xiTxi} applies to
the effect that $\xi$-connected and $\T\xi$-connected subsets coincide. 

To give examples of accessibility spaces, recall that a topological
space $X$ is \emph{Fréchet-Urysohn} if whenever $x\in\cl_{}A$ for
$x\in X$ and $A\subset X$, there is a sequence $\{x_{n}\}_{n=1}^{\infty}$
on $A$ converging to $x$. Fréchet-Urysohn spaces in which sequences
have unique limits are accessibility, e.g., \cite{why,ChigrMyn}.
Hence,
\begin{cor}
If $(X,\xi)$ is a convergence space for which $\T\xi$ is a Hausdorff
Fréchet-Urysohn topology, then $\xi$-connected and $\T\xi$-connected
subsets coincide.
\end{cor}

In this paper, we explore when $\xi$-connected and $\T\xi$-connected
subsets do or do not\emph{ }coincide.

\section{Convergences on finite sets and directed graph\protect\label{subsec:convergences-on-finite}}

Note that finitely deep convergences and pretopologies coincide on
finite sets. Moreover, they are entirely determined by the convergence
of principal ultrafilters. Denoting with an arrow $x\to y$ if $y\in\lm\{x\}^{\uparrow}$,
a pretopology on a finite set induces a directed graph (digraph) on
$X$, with a loop at each point, because of (\ref{eq:centeredconv}).
In this paper, we will use graphs to represent certain examples of
convergences on finite sets. We will omit the loops at each point
in order not to overburden pictures, because they are implicit. Thus,
for example, the usual Sierpinski topological space would be denoted
by 
\[
\xymatrix{0\ar[r] & 1}
\]
 and not 
\[
\xymatrix{0\ar[r]\ar@(dl,dr) & 1\ar@(dl,dr)}
\]

\bigskip{}

Note also that the basic example
\[
\xymatrix{a\ar[d]\\
b\ar[r] & c
}
\]
corresponds to the pretopology $\xi$ given by $\V_{\xi}(a)=\{a\}^{\uparrow}$,
$\V_{\xi}(b)=\{a,b\}^{\uparrow}$ and $\V_{\xi}(c)=\{b,c\}$ and is
not topological: an open set $O$ around $c$ must also contain $b$
because $\lim\{b\}^{\uparrow}\cap O\neq\emptyset$ hence $O\in\{b\}^{\uparrow}$.
By the same token, it must also contain $a$, so that $O=\{a,b,c\}$.
In particular, $c\in\lm_{\T\xi}\{a\}^{\uparrow}$ but $c\notin\lm_{\xi}\{a\}^{\uparrow}$.
In fact $\T\xi$ is given by the graph 
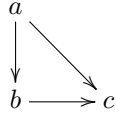
\[
\xymatrix{a\ar[d]\ar[dr]\\
b\ar[r] & c
}
\]
 and more generally, if $\xi$ is a pretopology on a finite set the
digraph for $\T\xi$ is the transitive closure of the digraph for
$\xi$. 

In a directed graph, we say a finite sequence $\{x_{n}\}_{n=1}^{n=k}$
of vertices is \emph{a path from $a$ to $b$ }if $x_{1}=a$, $x_{k}=b$
and $x_{n}\to x_{n+1}$ for every $n\in\{1,\ldots k-1\}$, that is,
in the convergence space interpretation $x_{n+1}\in\lm\{x_{n}\}^{\uparrow}$
for every $n\in\{1,\ldots k-1\}$. We say that a path from $a$ to
$b$ is \emph{in $A$ }if $x_{n}\in A$ for all $n\in\{1,\ldots k\}$.
In a (non-directed) graph, we may say ``a path between $a$ and $b$''
instead. A graph is connected if there is a path between any two vertices.
Recall that a directed graph is \emph{weakly connected }if its underlying
(non-directed) graph is connected, and \emph{connected} if there is
a directed path between every pair of vertices.
\begin{prop}
\label{prop:finiteconnected} A finite convergence space is connected
if and only if its underlying digraph is weakly connected.
\end{prop}

\begin{proof}
Suppose the underlying digraph is not weakly connected, that is, there
are $a$ and $b$ with no path from $a$ to $b$ in the underlying
(non-directed) graph. Then the graph connected component of $a$ is
closed and open in the convergence and does not contain $b$, so that
the convergence is not connected. Suppose conversely that the convergence
is not connected, so that there are two non-empty disjoint clopen
subsets. Pick $a$ and $b$ in these two disjoint sets. There is no
path between $a$ and $b$ in the underlying undirected graph, and
thus the underlying digraph is not weakly connected. 
\end{proof}
Note that this means that connectedness only depends on the underlying
non-directed graph. Hence, a finite convergence space $(X,\xi)$ is
connected if and only if its \emph{reciprocal modification} $r\xi$,
given by 
\begin{equation}
t\in\lm_{r\xi}\{x\}^{\uparrow}\iff t\in\lm_{\xi}\{x\}^{\uparrow}\text{ or }x\in\lm_{\xi}\{t\}^{\uparrow},\label{eq:reciprocal}
\end{equation}
which has the same underlying (non-directed) graph, is connected. 

\section{Reciprocal modification and connectedness}

The observation made on finite spaces that connectedness only depends
on the reciprocal modification extends to general spaces. First, let
us extend the definition of $r$, as a finite convergences of finite
depth only need be defined on principal ultrafilters, but $r\xi$
needs to be defined for general filter on an infinite $X$. Given
a convergence $\xi$ on $X$, define its \emph{reciprocal modification
}$r\xi$ by:
\begin{equation}
t\in\lm_{r\xi}\F\iff\begin{cases}
t\in\lm_{\xi}\F & \F\notin\mathbb{F}_{0}X\cap\mathbb{U}X\\
t\in\lm_{\xi}\{x\}^{\uparrow}\text{ or }x\in\lm_{\xi}\{t\}^{\uparrow} & \F=\{x\}^{\uparrow}
\end{cases}.\label{eq:reciprocalfull}
\end{equation}

\begin{prop}
\label{prop:rxiproperties} The reciprocal modification $r$ is a
concrete reflector which commutes with subspaces and for every convergence
space $(X,\xi)$, $\xi$-clopen and $r\xi$-clopen subsets of $X$
coincide. 
\end{prop}

\begin{proof}
The modification $r$ is order-preserving ($\xi\leq\tau\then r\xi\leq r\tau$),
contractive ($r\xi\leq\xi$), and idempotent ($r(r\xi)=r\xi$), hence
it is a projector in the sense of \cite{DM.book}. Moreover, if $f\in\cont(\xi,\tau)$
and $t\in\lm_{r\xi}\F$ then $t\in\lm_{\tau}\F\subset\lim_{r\tau}\F$
if $\F$ is not a principal ultrafilter. If it is, say $\F=\{x\}^{\uparrow}$,
then $t\in\lm_{\xi}\{x\}^{\uparrow}$ or $x\in\lm_{\xi}\{t\}^{\uparrow}$.
Hence, $f(t)\in\lm_{\tau}\{f(x)\}^{\uparrow}$ or $f(x)\in\lm_{\tau}\{f(t)\}^{\uparrow}$,
so that $f(t)\in\lm_{r\tau}\{f(x)\}^{\uparrow}$. Hence $r$ is also
a functor, hence a reflector. In particular, if $A\subset X$ then
$r(\xi_{|A})\geq(r\xi)_{|A}$ and we only need to verify the reverse
inequality. Suppose $a\in\lm_{r\xi}\F$ where $A\in\F$. If $\F$
is not principal then $a\in\lm_{\xi_{|A}}\F\subset\lim_{r(\xi_{|A})}\F$.
Else, $\F=\{b\}^{\uparrow}$ where $b\in A$ and $a\in\lm_{\xi}\{b\}^{\uparrow}$
or $b\in\lm_{\xi}\{a\}^{\uparrow}$, so that $a\in\lm_{r(\xi_{|A})}\F$.

As $\xi\geq r\xi$, $\T\xi\geq\T(r\xi)$, so that every $r\xi$-clopen
subset is also $\xi$-clopen. Let $U$ be a non-empty $\xi$-clopen
subset of $X$. To see that $U$ is $r\xi$-open, let $\lm_{r\xi}\F\cap U\neq\emptyset$.
If $\F\notin\mathbb{F}_{0}X\cap\mathbb{U}X$, $\lm_{\xi}\F\cap U\neq\emptyset$
and thus $U\in\F$ because $U$ is $\xi$-open. Else, $\F=\{x\}^{\uparrow}$
and there is $t\in U\cap\lm_{r\xi}\{x\}^{\uparrow}$, that is, either
$t\in\lm_{\xi}\{x\}^{\uparrow}$ and $U\in\{x\}^{\uparrow}=\F$, or
$x\in\lm_{\xi}\{t\}^{\uparrow}$ and $x\in\adh_{\xi}U\subset U$,
that is, $U\in\{x\}^{\uparrow}=\F$. Hence, $U$ is $r\xi$-open.
To see that $U$ is $r\xi$-closed, let $U\in\F$ and $t\in\lm_{r\xi}\F$.
If $\F\notin\mathbb{F}_{0}X\cap\mathbb{U}X$, $t\in\lm_{\xi}\F\subset\cl_{\xi}U=U$.
Else, $\F=\{x\}^{\uparrow}$ for som $x\in U$ and either $t\in\lm_{\xi}\{x\}^{\uparrow}\subset\cl_{\xi}U=U$
or $x\in\lm_{\xi}\{t\}^{\uparrow}$ and $U\in\{t\}^{\uparrow}$ because
$U$ is $\xi$-open, so that $t\in U$. Hence, $U$ is $r\xi$-closed.
\end{proof}
\begin{cor}
\label{cor:connectedrxi} A subset $A$ of a convergence space $(X,\xi)$
is $\xi$-connected if and only if it is $r\xi$-connected.
\end{cor}

\begin{proof}
As $\xi\geq r\xi$, $A$ is $r\xi$-connected whenever it is $\xi$-connected.
Conversely, assume that $A$ is not $\xi$-connected, that is, there
is a non-empty proper $\xi_{|A}$-clopen subsets $C$ of $A$. In
view of Proposition \ref{prop:rxiproperties}, $C$ is also $r(\xi_{|A})$-clopen
and $r(\xi_{|A})=(r\xi)_{|A}$, hence $A$ is not $r\xi$-connected.
\end{proof}
Consider the dual Alexandroff pretopologies $\xi^{*}$ and $\xi^{\bullet}$
associated with a convergence $\xi$ on $X$ and defined by 
\[
\adh_{\xi^{*}}A:=\{t\in X:\lm_{\xi}\{t\}^{\uparrow}\cap A\neq\emptyset\},
\]
and 
\[
\adh_{\xi^{\bullet}}A:=\bigcup_{a\in A}\lm_{\xi}\{a\}^{\uparrow},
\]
respectively (their topological modifications are considered in \cite{DM.uK,myn.completeness}).
Note that $B\cap\adh_{\xi^{\bullet}}A\neq\emptyset$ if and only if
$A\cap\adh_{\xi^{*}}B\neq\emptyset$ and thus a subset of $X$ is
$\xi^{*}$-open if and only if it is $\xi^{\bullet}$-closed. Of course,
if $\xi$ is $T_{1}$ then both $\xi^{*}$ and $\xi^{\bullet}$ are
discrete. 
\begin{prop}
\label{prop:rxiadherence} If $(X,\xi)$ is a convergence space and
$A\subset X$ then 
\[
\adh_{r\xi}A=\adh_{\xi}A\cup\adh_{\xi^{*}}A,
\]
and more generally,
\[
\adh_{r\xi}^{\alpha}A=\adh_{\xi}^{\alpha}A\cup\adh_{\xi^{*}}^{\alpha}A,
\]
 for every ordinal $\alpha$. In particular, 
\[
\cl_{r\xi}A=\cl_{\xi}A\cup\cl_{\xi^{*}}A.
\]
\end{prop}

\begin{proof}
As $\xi\geq r\xi$, $\adh_{\xi}A\subset\adh_{r\xi}A$ and if $x\in\adh_{\xi^{*}}A$
then there is $t\in\lim_{\xi}\{x\}^{\uparrow}\cap A$, so that $x\in\lm_{r\xi}\{t\}^{\uparrow}\subset\adh_{r\xi}A$.
Hence, $\adh_{\xi^{*}}A\subset\adh_{r\xi}A$.

Conversely, if $t\in\lm_{r\xi}A$ there is a filter $\F$ with $A\in\F$
and $t\in\lm_{r\xi}\F$. If $\F\notin\mathbb{F}_{0}X\cap\mathbb{U}X$,
$t\in\lm_{\xi}\F\subset\adh_{\xi}A$. Else, there is $x\in A$ with
$\F=\{x\}^{\uparrow}$ and $t\in\lm_{r\xi}\{x\}^{\uparrow}$, so that
either $t\in\lm_{\xi}\{x\}^{\uparrow}\subset\adh_{\xi}A$ or $x\in\lm_{\xi}\{t\}^{\uparrow}$
and $t\in\adh_{\xi^{*}}A$.

Assume that $\adh_{r\xi}^{\beta}A=\adh_{\xi}^{\beta}A\cup\adh_{\xi^{*}}^{\beta}A$
for every $\beta<\alpha$. Then 
\begin{eqnarray*}
\adh_{r\xi}^{\alpha}A & = & \adh_{r\xi}\Big(\bigcup_{\beta<\alpha}\adh_{r\xi}^{\beta}A\Big)\\
 & = & \adh_{r\xi}\Big(\bigcup_{\beta<\alpha}\adh_{\xi}^{\beta}A\cup\adh_{\xi^{*}}^{\beta}A\Big)\\
 & = & \adh_{r\xi}\Big((\bigcup_{\beta<\alpha}\adh_{\xi}^{\beta}A)\cup(\bigcup_{\beta<\alpha}\adh_{\xi^{*}}^{\beta}A)\Big)\\
 & = & \adh_{r\xi}(\bigcup_{\beta<\alpha}\adh_{\xi}^{\beta}A)\cup\adh_{r\xi}(\bigcup_{\beta<\alpha}\adh_{\xi^{*}}^{\beta}A)\\
 & = & \adh_{\xi}^{\alpha}A\cup\adh_{\xi^{*}}^{\alpha}A.
\end{eqnarray*}
\end{proof}

\section{Enclosing sets}

Let us say that a subset $S$ of a convergence space $(X,\xi)$ \emph{encloses
}a connected subspace $A$ if $B$ is connected whenever $A\subset B\subset S$.

Proposition \ref{prop:adhsandwich} states that the adherence of a
connected set encloses that set. In view of Corollary \ref{cor:connectedrxi}
and Proposition \ref{prop:rxiadherence}, so does its $r\xi$-adherence,
which is the union of its $\xi$-adherence and its $\xi^{*}$-adherence.
It turns out to be the largest possible enclosing set.

\begin{lem}
\label{lem:generalenclose} Let $A$ be a non-empty connected subset
of a convergence space $(X,\xi)$. Then $S\subset X$ encloses $A$
if and only if $A\cup\{p\}$ is connected whenever $p\in S$.
\end{lem}

\begin{proof}
$\then$ is obvious, using the enclosing property with $A\subset A\cup\{p\}\subset S$.

$\Longleftarrow$: Let $A\subset C\subset S$. Consider the family
$\{A\cup\{c\}:c\in C\}$ of connected subsets of $X$. As $(A\cup\{c\})\cap A\neq\emptyset$
for every $c\in C$, we conclude from Proposition \ref{prop:unionconnected}
that $\bigcup_{c\in C}A\cup\{c\}=C$ is also connected.
\end{proof}
As a result of Lemma \ref{lem:generalenclose}, given a non-empty
connected set $A$, the \emph{enclosure of $A$} defined by 
\[
e(A)=\{x\in X:A\cup\{x\}\text{ is connected}\}
\]
is the largest (for inclusion) subset of $X$ that encloses $A$,
so that
\begin{equation}
S\text{ encloses }A\iff S\subset e(A).\label{eq:eofAbiggest}
\end{equation}

\begin{thm}
If $A$ is a non-empty connected subset of a convergence space $(X,\xi)$,
then
\[
e(A)=\adh_{\xi}A\cup\adh_{\xi^{*}}A=\adh_{r\xi}A.
\]
\end{thm}

\begin{proof}
By Proposition \ref{prop:adhsandwich} and Corollary \ref{cor:connectedrxi},
$\adh_{r\xi}A\subset e(A)$. 

Suppose $t\notin\adh_{\xi}A$ and $t\notin\adh_{\xi^{*}}A$. Then
$A\cup\{t\}$ is not connected. Indeed, $A$ is closed in $\xi_{|A\cup\{t\}}$
and is also open, for $\lim_{\xi_{|A\cup\{t\}}}\{t\}^{\uparrow}\cap A=\emptyset$.
Hence $t\notin e(A)$.
\end{proof}
\begin{rem}
\label{rem:enotidempotent} As a principal adherence operator, $\adh_{r\xi}$
is additive (commutes with finite unions), but it may fail to be idempotent.
This is easily seen in the $T_{1}$ case where $\adh_{r\xi}=\adh_{\xi}$,
picking a standard non-topological Hausdorff pretopology like the
Féron-cross pretopology of \cite[Example V.6.1]{DM.book}. This pretopology
on $\mathbb{R}^{2}$ is given by the vicinity filters $\V(x,y)=\{B_{d}((x,y),r):r>0\}^{\uparrow}$,
where $B_{d}((x,y),r)=\{x\}\times(y-r,y+r)\cup(x-r,x+r)\times\{y\}$.

In that example, the set $B=\{(x,y)\in\mathbb{R}^{2}:0<x<y<1\}$ is
connected and $e(B)=\adh B=\{(x,y)\in\mathbb{R}^{2}:0\leq x\leq y\leq1\}\setminus\{(0,0),(0,1),(1,1)\}$
and $e(e(B))=\adh^{2}B=\{(x,y)\in\mathbb{R}^{2}:0\leq x\leq y\leq1\}$
is closed.
\end{rem}

\begin{cor}
\label{cor:adhofconnectedisclosed} If $(X,\xi)$ is a $T_{1}$ convergence
space and $A$ is a non-empty connected subset, then $e(A)=\adh_{\xi}A$.
In particular, $\cl_{\xi}A$ encloses $A$ if and only if $\adh_{\xi}A=\cl_{\xi}A$.
\end{cor}

Let us say that a convergence space $(X,\xi)$ \emph{has the sandwich
property }if $\cl_{\xi}A$ encloses $A$ for every connected subset
$A$ of $X$. Note that for a convergence space $(X,\xi),$ the functorial
condition that $\S_{0}\xi=\T\xi$ is equivalent to the fact that $\adh_{\xi}A=\cl_{\xi}A$
for all subsets $A$ of $X$, so that, in view of Corollary \ref{cor:adhofconnectedisclosed},
every convergence $\xi$ satisfying this condition also has the sandwich
property, but this condition is not necessary for the sandwich property:
\begin{example}[A $T_{1}$ convergence with the sandwich property for which $\S_{0}\xi>\T\xi$]
\label{exa:bisequence} The usual bisequence pretopology (see, e.g.,
\cite[Example V.4.6]{DM.book}) is non-topological, hence satisfies
$\xi=\S_{0}\xi>\T\xi$ and it is $T_{1}$ (in fact Hausdorff). Its
only non-empty connected subsets are the singletons, which are closed,
so that it has the sandwich property. Let us describe this example
more explicitly for future use: on $X=\{x_{n,k}:n,k\in\omega\}\cup\{x_{n}:n\in\omega\}\cup\{x{}_{\infty}\}$
consider the pretopology $\xi$ in which each $x_{n,k}$ is isolated,
that is, $\V(x_{n,k})=\{x_{n,k}\}^{\uparrow}$, and $\V(x_{n})=\{\{x_{n,k}:k\geq p\}\cup\{x_{n}\}:p\in\omega\}^{\uparrow}$
for every $n\in\omega$, and $\V(x_{\infty})=\{\{x_{n}:n\geq p\}\cup\{x_{\infty}\}:p\in\omega\}^{\uparrow}$. 
\end{example}

On the other hand, $T_{1}$ is necessary in Corollary \ref{cor:adhofconnectedisclosed}:
\begin{example}[A convergence space with the sandwich property and a connected subset
with non-closed adherence ]

Let us consider a triangular directed graph, interpreted as a pretopology.
Namely, consider $X=\{a,b,c\}$ with the pretopology determined by
$\lim\{a\}^{\uparrow}=\{a,b\}$, $\lim\{b\}^{\uparrow}=\{b,c\}$ and
$\lim\{c\}^{\uparrow}=\{c,a\}$.

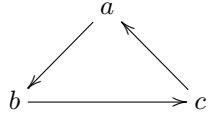
\[
\xymatrix{ & a\ar[dl]\\
b\ar[rr] &  & c\ar[ul]
}
\]

It is easily checked that every subset is connected (observe that
$r\xi$ is the antidiscrete topology!), hence this space has the sandwich
property. However, $\adh\{a\}=\{a,b\}$ while $\cl\{a\}=\{a,b,c\}$. 
\end{example}

\section{$\protect\T$-subspaces}

Let us call a subset $A$ of a convergence space $(X,\xi)$ a $\T$-\emph{subspace}
if $(\T\xi)_{|A}=\T(\xi_{|A})$. In view of (\ref{eq:Tsubspace}),
this is equivalent to 
\begin{equation}
\T(\xi_{|A})\leq(\T\xi)_{|A},\label{eq:Trealsubspace}
\end{equation}
that is, every $\xi_{|A}$-closed set is $(\T\xi)_{|A}$-closed. In
other words, $A$ is a $\T$-subspace if and only if
\begin{equation}
\adh_{\xi}B\cap A=B\then\cl_{\xi}B\cap A=B,\label{eq:Tsubspaceadhcl}
\end{equation}
 for every subset $B$ of $A$.

In view of the discussion following Proposition \ref{prop:S0xiTxi},
it is clear that:
\begin{prop}
\label{prop:Tsubconnected}If $(X,\xi)$ is a convergence space and
$A\subset X$ is a $\T$-subspace, then $A$ is $\xi$-connected if
and only if it is $\T\xi$-connected.
\end{prop}

Here are two prototypic examples of subset that fails to be a $\T$-subspace:
\begin{example}[a non $\T$-subspace (finite example)]
\label{exa:nonTsubfinite} Let $X=\{a,b,c\}$ with the ``line''
pretopology given by $\lim\{a\}^{\uparrow}=\{a,b\}$, $\lim\{b\}^{\uparrow}=\{b,c\}$
and $\lim\{c\}^{\uparrow}=\{c\}$, that is, $\V(a)=\{a\}^{\uparrow}$,
$\V(b)=\{a,b\}^{\uparrow}$ and $\V(c)=\{b,c\}^{\uparrow}$.
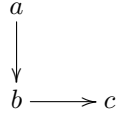
\[
\xymatrix{a\ar[d]\\
b\ar[r] & c
}
\]

Consider the subset $A=\{a,c\}$. Its subset $\{a\}$ is $\xi_{|A}$-closed
but $c\in\cl_{(\T\xi)_{|A}}\{a\}$, for the only $\xi$-open set containing
$c$ is $X$, so that $\{a\}$ is not $(\T\xi)_{|A}$-closed. Hence
$A$ is not a $\T$-subspace.
\end{example}

\begin{example}[a non $\T$-subspace (infinite Hausdorff example)]
\label{exa:noTsubinfinite} Consider the standard bisequence pretopology
of Example \ref{exa:bisequence} and its subset $A=\{x_{n,k}:n,k\in\omega\}\cup\{x_{\infty}\}$.
Note that its subset $\{x_{n,k}:n,k\in\omega\}$ is $\xi_{|A}$-closed
but not $(\T\xi)_{|A}$-closed, for $x_{\infty}\in\cl_{(\T\xi)_{|A}}\{x_{n,k}:n,k\in\omega\}$,
so that $A$ is not a $\T$-subspace.
\end{example}

These two basic examples illustrate the easy fact that:
\begin{prop}
\label{prop:notTsub} If $(X,\xi)$ is a convergence space, $S\subset X$
and $x\in\cl_{\xi}S\setminus\adh_{\xi}S$ then $S\cup\{x\}$ is not
a $\T$-subspace.
\end{prop}

\begin{proof}
Indeed, if $A=S\cup\{x\}$ then $S$ is $\xi_{|A}$-closed but not
$(\T\xi)_{|A}$-closed as $x\in\cl_{(\T\xi)_{|A}}S$.
\end{proof}
\begin{cor}
If $(X,\xi)$ is a convergence space, the following are equivalent:
\begin{enumerate}
\item $\S_{0}\xi=\T\xi$;
\item $\xi$ has topological defect 1;
\item Every subset of $X$ is a $\T$-subspace.
\end{enumerate}
\end{cor}

\begin{proof}
$(1)\iff(2)$ is by definition and we have seen that $(1)$ implies
$(3)$ because $\S_{0}(\xi_{|A})=(\S_{0}\xi)_{|A}$ \cite[Corollary XIV.3.9]{DM.book}.
Finally, $(3)\then(1)$ by Proposition \ref{prop:notTsub}, as $\cl_{\xi}S=\adh_{\xi}S$
for every $S\subset X$ if all subsets are $\T$-subspaces.
\end{proof}
Singletons are always $\T$-subspaces, because there's only one convergence
on a singleton.
\begin{rem}
\label{rem:Tsubnotclosedunderoperations} Note that a union of two
$\T$-subspaces may fail to be a $\T$-subspace. For instance, the
two singletons $\{a\}$ and $\{c\}$ in Example \ref{exa:nonTsubfinite}
are $\T$-subspaces but $\{a\}\cup\{c\}$ is not. Similarly, the complement
of a $\T$-subspace may fail to be a $\T$-subspace. For instance,
The singleton $\{b\}$ in Example \ref{exa:nonTsubfinite} is a $\T$-subspace,
but its complement $\{a,c\}$ is not. Moreover the intersection of
two $\T$-subspaces may fail to be a $\T$-subspace, as the example
below illustrates.
\end{rem}

\begin{example}[the intersection of two $\T$-subspaces may fail to be a $\T$-subspace]
\label{exa:intersectionofTsub} Consider the pretopology on $X=\{a,b,c,d\}$
given by $\lim\{a\}^{\uparrow}=\{a,b,d\}$, $\lim\{b\}^{\uparrow}=\{b,c\}$,
$\lim\{c\}^{\uparrow}=\{c\}$ and $\lim\{d\}^{\uparrow}=\{d,c\}$:
\[
\xymatrix{a\ar[r]\ar[d] & b\ar[d]\\
d\ar[r] & c
}
\]

The subsets $A=\{a,b,c\}$ and $B=\{a,c,d\}$ are $\T$-subspaces
as $\T(\xi_{|A})=(\T\xi_{|A})$ is given by 
\[
\xymatrix{a\ar[r]\ar[rd] & b\ar[d]\\
 & c
}
\]
and similarly for $B$, while $A\cap B=\{a,c\}$ is not a $\T$-subspace,
as $\T(\xi_{|A\cap B})=\xi_{|A\cap B}$ is the discrete topology,
while $(\T\xi)_{|A\cap B}$ is an homeomorphic copy of the Sierpinski
topology 
\[
\xymatrix{a\ar[r] & c}
\]
On the other hand,
\end{example}

\begin{lem}
\label{lem:openorclosedT} A subset of a convergence space that is
either closed or open is a $\T$-subspace. 
\end{lem}

\begin{proof}
Let $A\subset X$ where $(X,\xi)$ is a convergence space. To show
that $A$ is a $\T$-subspace, we show (\ref{eq:Trealsubspace}). 

Suppose $A$ is closed and $C\subset A$ is $\xi_{|A}$-closed and
let $\F\in\mathbb{F}X$ with $C\in\F$ and $x\in\lim_{\xi}\F$. As
$A$ is closed, $x\in A$, hence $x\in\lm_{\xi_{|A}}\F$ and thus
$x\in C$. Therefore, $C$ is a $\xi$-closed subset of $A$.

Suppose $A$ is open $U\subset A$ is $\xi_{_{|A}}$-open, that is,
if $\F\in\mathbb{F}A$ and $\lim_{\xi}\F\cap U\neq\emptyset$ then
$U\in\F$. If now $\G\in\mathbb{F}X$ and $\lim_{\xi}\G\cap U\neq\emptyset$
then $\lim_{\xi}\G\cap A\neq\emptyset$ and $A\in\G$ because $A$
is open, hence $U\in\G$ and $U$ is $\xi$-open.
\end{proof}
\begin{rem}
\label{rem:neitherclosednoropen}Note that we can modify the pretopology
of Example \ref{exa:intersectionofTsub} to ensure that the sets $A$
and $B$ are still $\T$-subspaces with a non-$\T$ intersection,
but that they are neither open nor closed (nor singletons), by taking
the reciprocal modification:
\[
\xymatrix{a\ar[r]\ar[d] & b\ar[d]\ar[l]\\
d\ar[r]\ar[u] & c\ar[u]\ar[l]
}
\]

Hence there are non-singleton $\T$-subspaces that are neither closed
nor open, even when $\S_{0}\xi\neq\T\xi$.
\end{rem}

In the finite case, $\T$-subspaces find a somewhat more concrete
characterization than (\ref{eq:Tsubspaceadhcl}):
\begin{prop}
Let $(X,\xi)$ be a finite convergence space. Then $A\subset X$ is
a $\T$-subspace if and only if for every $a,b\in A$, if there is
a path from $a$ to $b$, there is also a path from $a$ to $b$ in
$A$.
\end{prop}

\begin{proof}
Suppose $A$ is not a $\T$-subspace, that is, $\T(\xi_{|A})\nleq(\T\xi)_{|A}$.
In other words, there is $B\subset A$ that is $\xi_{|A}$-closed
but not $(\T\xi)_{|A}$-closed, so that there is $b\in B$ with $a\in\cl_{\xi}\{b\}\cap(A\setminus B)$.
In other words, there is a path from $b$ to $a$. Because $B$ is
$\xi_{|A}$-closed, there is no such path in $A$. Conversely, suppose
that $A$ is a $\T$-subspace and that there is a path from $a$ to
$b$ where $a,b\in A$. Hence $b\in\cl_{\xi}\{a\}$ and $\T(\xi_{|A})=(\T\xi)_{|A}$,
so that $b\in\cl_{\xi_{|A}}\{a\}$. Therefore, there is a path from
$a$ to $b$ in $A$.
\end{proof}

\section{Sandwiched sets}

In a convergence space without the sandwich property, it makes sense
to study the connectedness of sandwiched subsets, where $B$ is called
\emph{sandwiched }if there is a connected subspace $A$ with $A\subset B\subset\cl_{\xi}A$.
We may then say that $B$ is \emph{sandwiched by $A$. }Note that
if $B$ is sandwiched by $A$, then $\cl_{\xi}B=\cl_{\xi}A$ is connected
by Proposition \ref{prop:closedconn}. Moreover, since $A$ is $\xi$-connected,
it is also $\T\xi$-connected, hence every set between $A$ and its
closure, in particular $B$, is $\T\xi$-connected by the standard
topological result. In other words, every sandwiched set is $\T\xi$-connected,
but may fail to be $\xi$-connected (as there are spaces without the
sandwich property). Since a convergence space is connected if and
only if its topological modification is connected, a $\T\xi$-connected
$\T$-subspace of a convergence space $(X,\xi)$ is also $\xi$-connected.
In fact, we have
\begin{align*}
\xi\text{-connected} & \then & \text{\ensuremath{\xi}-sandwiched} & \then & \T\xi\text{-connected}\\
\Updownarrow &  & \Downarrow &  & \Downarrow\\
r\xi\text{-connected} & \then & r\xi\text{-sandwiched} & \then & \T(r\xi)\text{-connected}
\end{align*}

Though $\xi$-connected subsets and $r\xi$-connected subsets coincide,
there are $r\xi$-sandwiched sets that are not $\xi$-sandwiched and
$\T(r\xi)$-connected sets that are not $\T\xi$-connected, as we
shall see in Example \ref{exa:Txiconnnotsandwiched} below, so that
the one-directional vertical arrows cannot be reversed. 

On the other hand, all three notions in the first row coincide among
$\T$-subspaces (and in the second row for $\T$-subspaces for $r\xi$),
but no horizontal arrow can be reversed in general. Indeed, we have
seen that there are sandwiched sets that are not $\xi$-connected.
Moreover, we may have a $\T\xi$-connected set that is not sandwiched:
\begin{example}[A $\T\xi$-connected set that is not sandwiched]
\label{exa:Txiconnnotsandwiched} Consider the pretopology on $X=\{a,b,c,d\}$
given by $\lim\{a\}^{\uparrow}=\{a,b\}$, $\lim\{b\}^{\uparrow}=\{b\}$,
$\lim\{c\}^{\uparrow}=\{c,b\}$ and $\lim\{d\}^{\uparrow}=\{c,d\}$:
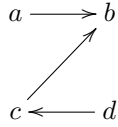
\[
\xymatrix{a\ar[r] & b\\
c\ar[ur] & d\ar[l]
}
\]

Then $A=\{a,b,d\}$ is $\T\xi$-connected, as $\T\xi$ is given by
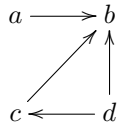
\[
\xymatrix{a\ar[r] & b\\
c\ar[ur] & d\ar[l]\ar[u]
}
\]

but it is not sandwiched, as the only $\xi$-connected subsets of
$A$ are $\{d\}$ and $A\not\subset\cl_{\xi}\{d\}=\{d,c,b\}$ and
$\{a,b\}$, which is $\xi$-closed.

Note also that $\T(r\xi)$ is the antidiscrete topology, so every
subset is $\T(r\xi)$-connected and every non-empty subset is $r\xi$-sandwiched
because thus the $r\xi$-closure of a singleton is $X$. However,
$\{a,c\}$ is not $\T\xi$-connected, and $A$ is not $\xi$-sandwiched.
\end{example}

To summarize:
\begin{prop}
\label{prop:basicsandwichedset}Let $(X,\xi)$ be a convergence space.
\begin{enumerate}
\item Every $\xi$-connected set is sandwiched.
\item Every sandwiched set is $\T\xi$-connected. 
\item A $\T\xi$-connected (in particular, sandwiched) $\T$-subspace is
$\xi$-connected.
\item There are sandwiched sets that are not $\xi$-connected and $\T\xi$-connected
subsets that are not sandwiched.
\item If $\bigcup_{\beta<\alpha}\adh_{r\xi}^{\beta}A\subset B\subset\adh_{r\xi}^{\alpha}A$
for some ordinal $\alpha$ and $\xi$-connected set $A$, then $B$
is $\xi$-connected.
\end{enumerate}
\end{prop}

In particular, every iterated adherence of a connected set is a \emph{connected
}sandwiched set.

\bibliographystyle{plain}

\begin{thebibliography}{1}
	
	\bibitem{adamek1990abstract}
	Ji{\v{r}}{\'\i} Ad{\'a}mek, Horst Herrlich, and George Strecker.
	\newblock {\em Abstract and concrete categories}.
	\newblock Wiley-Interscience, 1990.
	
	\bibitem{ChigrMyn}
	F.~Chigr and F.~Mynard.
	\newblock The convergence-theoretic approach to weakly first countable spaces
	and symmetrizable spaces.
	\newblock {\em Mathematica Slovaca}, 69(1):185--198, 2019.
	
	\bibitem{DG}
	S.~Dolecki and G.~H. Greco.
	\newblock Topologically maximal pretopologies.
	\newblock {\em Studia Math.}, {\bf 77}:265--281, 1984.
	
	\bibitem{DM.uK}
	S.~Dolecki and F.~Mynard.
	\newblock Hyperconvergences.
	\newblock {\em Appl. Gen. Top.}, 4(2):391--419, 2003.
	
	\bibitem{DM.book}
	S.~Dolecki and F.~Mynard.
	\newblock {\em Convergence {F}oundations of {T}opology}.
	\newblock World Scientific, 2016.
	
	\bibitem{dolecki1998topologically}
	S.~Dolecki and M.~Pillot.
	\newblock Topologically maximal convergences, accessibility, and covering maps.
	\newblock {\em Mathematica Bohemica}, 123(4):371--384, 1998.
	
	\bibitem{bryantesis}
	B.G.~Castro Herrej{\'o}n.
	\newblock Topolog{\'\i}a sin topolog{\'\i}as: una introducci{\'o}n al estudio
	de la convergencia.
	\newblock tesis de licenciatura, UNAM, 2023.
	
	\bibitem{myn.completeness}
	F.~Mynard.
	\newblock ({U}ltra-) completeness numbers and (pseudo-) paving numbers.
	\newblock {\em Topology and its Applications}, 256:86--103, 2019.
	
	\bibitem{why}
	G.~T. Whyburn.
	\newblock Accessibility spaces.
	\newblock {\em Proc. Amer. Math. Soc.}, {\bf 24}:181--185, 1970.
	
\end{thebibliography}

\end{document}